\pdfoutput=1

\documentclass[12pt]{amsart}
\usepackage[dvips]{graphicx}
\usepackage{color}
\usepackage{amssymb} 
\usepackage[english]{babel}
\usepackage{amsmath,amsthm}
\usepackage{amsfonts}
\usepackage{latexsym}
\usepackage{amscd}
\usepackage[latin1]{inputenc}
\usepackage{enumerate}
\usepackage{afterpage}
\usepackage[all]{xy}
\usepackage{mathrsfs}
\usepackage{array}
\usepackage{tikz}

\newcommand{\C}{\mathbb{C}}

\newcommand{\R}{\mathbb{R}}

\newcommand{\Z}{\mathbb{Z}}

\newcommand{\PPP}{\mathbb{P}}

\renewcommand{\to}{\longrightarrow}

\newtheorem{Theorem}{Theorem}[section]
\newtheorem{Definition}[Theorem]{Definition}
\newtheorem{Lemma}[Theorem]{Lemma}
\newtheorem{Proposition}[Theorem]{Proposition}
\newtheorem{Corollary}[Theorem]{Corollary}
\newtheorem{Remark}[Theorem]{Remark}
\newtheorem{Example}[Theorem]{Example}

\DeclareMathOperator{\Hom}{Hom}

\DeclareMathOperator{\Cone}{Cone}

\newcommand{\A}{\mathcal{A}}
\newcommand{\T}{\mathbb{T}}
\newcommand{\LL}{\mathcal{L}}
\DeclareMathOperator{\Exp}{Exp}
\DeclareMathOperator{\Adm}{Adm}
\DeclareMathOperator{\NonAdm}{NonAdm}

\DeclareMathOperator{\Tess}{\T^{ess}}
\DeclareMathOperator{\Admess}{Adm^{ess}}
\DeclareMathOperator{\NonAdmess}{NonAdm^{ess}}

\begin{document}

\title{On the admissibility of certain local systems}

\begin{abstract}

A rank one local system on the complement of a hyperplane arrangement is 
said to be admissible if it satisfies certain non-positivity condition at 
every resonant edges. It is known that the cohomology of admissible local system can be 
computed combinatorially. In this paper, we study the structure of the 
set of all non-admissible local systems in the character torus. 
We prove that 
the set of non-admissible local systems forms a union of subtori. 
The relations with characteristic varieties are also discussed. 
\end{abstract}

\author{Shaheen Nazir, Michele Torielli, and Masahiko Yoshinaga} 

\address{\begin{flushleft}Shaheen Nazir, Department of Mathematics, LUMS School of Science and Engineering, U-Block, D.H.A, Lahore, Pakistan.\\
\emph{E-mail address:} \rm {\texttt{shaheen.nazeer@lums.edu.pk} }  \end{flushleft}}
\address{\begin{flushleft}Michele Torielli, Department of Mathematics, Hokkaido University, Kita 10, Nishi 8, Kita-Ku, Sapporo 060-0810, Japan.\\
\emph{E-mail address:} \rm {\texttt{torielli@math.sci.hokudai.ac.jp} }\end{flushleft}}
\address{\begin{flushleft}Masahiko Yoshinaga, Department of Mathematics, Hokkaido University, Kita 10, Nishi 8, Kita-Ku, Sapporo 060-0810, Japan.\\
\emph{E-mail address:} \rm {\texttt{yoshinaga@math.sci.hokudai.ac.jp} } \end{flushleft}}
\date{\today}
\maketitle

\tableofcontents

\section{Introduction}
Let $\A=\{H_1, \dots, H_n\}$ be a finite set of 
hyperplanes in $\PPP_\C^\ell$. 
The hyperplane arrangement $\A$ 
determines a poset $L(\A)$ of subspaces obtained as intersections of 
hyperplanes in $\A$. The combinatorial structure of $L(\A)$ 
is deeply related to the topology of $M(\A)$. 
In \cite{orlik1980combinatorics} Orlik and Solomon proved that 
the cohomology ring $H^*(M(\A), \Z)$ can be described in terms 
of combinatorial structures of $L(\A)$. However the homotopy type of 
$M(\A)$ can not be determined by $L(\A)$. Indeed, 
in \cite{rybnikov2011fundamental} Rybnikov proved that 
the fundamental group $\pi_1(M(\A))$ can not be determined 
by $L(\A)$. The combinatorial decidability of other 
topological invariants is still widely open. One of 
such invariant is rank one local system cohomology groups over 
$M(\A)$, which is originally motivated by 
hypergeometric functions \cite{aom-kit, gel-hyp}. 
Rank one local systems are parametrized by points in 
the character torus $\T(\A)=\Hom(\pi_1(M(\A)),\C^*)$. 
For generic $t\in\T(\A)$, it is proved (\cite{koh}) that 
$H^i(M(\A), \mathcal{L}_t)=0$ for $i\neq \ell$, where 
$\mathcal{L}_t$ is a rank one local system corresponding 
to $t\in\T(\A)$. Furthermore, Esnault-Schechtman-Viehweg 
\cite{esnault1992cohomology} (and \cite{schechtman1994local}) proved that 
for a local system $\LL_t$ such that the residue of associated logarithmic 
connection at each hyperplane is not 
a positive integer, 
the cohomology group $H^i(M(\A), \LL_t)$ can be computed 
by using the cochain complex, so-called the Aomoto complex, 
defined on the graded module $H^*(M(\A), \C)$, see \S\ref{sec:structure}. Such a local 
system is now called an admissible local system (see Definition 
\ref{defadlocsys}). 
Therefore, for an admissible local system $\LL_t$, 
the cohomology $H^*(M(\A), \LL_t)$ is combinatorially computable. 

For some arrangements, it has been proved that all rank one local 
systems are admissible (\cite{nazir2009admissible}). However, in general, 
the set of non-admissible local systems is non-empty. 
A natural strategy to study combinatorial decidability of 
local system cohomology groups is: 
(a) Determine the set of all non-admissible local systems 
in $\T(\A)$. 
(b) Compute the local system cohomology groups for 
non-admissible local systems. 
The purpose of this paper is related to the part (a) of the above strategy. 
We study the basic properties of the set of non-admissible local systems 
in the character torus $\T(\A)$. 

The paper is organized as follows. 
In \S\ref{sec:general}, we generalise the notion of admissible local systems. Consider the exponential map $\Exp\colon V=\C^n\to \mathbb{T}=(\C^*)^n,$ with kernel $\Lambda\cong\mathbb{Z}^n$. We introduce the notion 
``$\Phi$-admissibility'' on the 
algebraic torus $\T=(\C^*)^n$, for any finite set $\Phi\subset V^*$ of linear forms which preserve integral structure. We prove that the set of non-admissible 
points in $\T$ forms a union of subtori. We also give several conditions 
on $t\in\T$ to be admissible/non-admissible. 
In \S\ref{sec:structure}, we apply results from the previous section to 
the case of character torus of the complement complex line arrangements. 
In \S\ref{sec:characteristic}, we discuss the relation between non-admissible 
local systems and characteristic varieties. In particular, we prove that 
the local system corresponding to a point in the translated component 
(in the characteristic variety) is non-admissible. 
In \S\ref{sec:examples}, we describe several examples. 
\paragraph{\textbf{Acknowledgements}} The first author would like to  thank the Department of Mathematics and Faculty of Science of Hokkaido University
for kind hospitality during the initial stage of this work. The second author is supported by JSPS Fellowship for Foreign Researcher. The third author is supported by JSPS Grants-in-Aid for Scientific Research (C).

\section{General theory}
\label{sec:general}

Let $V\cong\C^n$ be a vector space and let $\mathbb{T}\cong(\C^*)^n$ be a complex torus. Consider the exponential mapping
$$
\Exp\colon V\to \mathbb{T},
$$
induced by the usual exponential function 
$\C\to\C^*$, $t\mapsto\Exp(t)=\exp(2\pi it)$, 
where $i=\sqrt {-1}$. Let $\Lambda:=\ker(\Exp)$. Note that $\Lambda$ is a lattice and hence $\Lambda\cong\mathbb{Z}^n$. Then we have that $V=\Lambda\otimes\C$ and define $V_\R=\Lambda\otimes\R\cong\R^n$.

Consider a set $\Phi:=\{\alpha_1,\dots,\alpha_r\}$ of 
linear maps $\alpha_i\colon V\to\C$ such that $\alpha_i(\Lambda)\subset\mathbb{Z}$ for all $i=1,\dots, r$. 
It is easily seen that there exists a character $\tilde{\alpha}_i\colon\mathbb{T}\to \C^*$ such that the following diagram is commutative
\begin{equation*}
\xymatrix{V \ar[d]^{\alpha_i}  \ar[r]^\Exp &\mathbb{T} \ar[d]^{\tilde{\alpha}_i} \\
 \C \ar[r]^{\Exp} &\C^*.} \label{defdiag} \end{equation*}
 \begin{Definition}\label{defphiadmiss} Consider $t\in\mathbb{T}$. It is called \emph{$\Phi$-admissible} if there exists $v\in V$ such that $\Exp(v)=t$ and $\alpha_i(v)\notin\mathbb{Z}_{>0}$ for all $\alpha_i\in\Phi$. 
 \end{Definition}
 \begin{Example} The unit $1\in\mathbb{T}$ is $\Phi$-admissible for all $\Phi$. In fact, $\Exp(0)=1$ and $\alpha_i(0)=0\notin\mathbb{Z}_{>0}$ for all $\alpha_i\in\Phi$.
 \end{Example}
Note that, for any given $\Phi$, we can write $\mathbb{T}=\Adm(\Phi)\sqcup \NonAdm(\Phi)$, where $\Adm(\Phi)$ is the set of $\Phi$-admissible elements and $\NonAdm(\Phi)$ is the set of non $\Phi$-admissible elements.
 \begin{Lemma}\label{lem:open} 
The set $\Adm(\Phi)$ is open in $\mathbb{T}$.
\end{Lemma}
\begin{proof} This is because the mapping $\Exp$ is a local homeomorphism and because the conditions of $\Phi$-admissibility are open conditions.
\end{proof}
\begin{Definition} Let $S\subset\Phi$ be a subset. Then we define $\mathbb{T}(S):=\{t\in\mathbb{T}~|~\tilde{\alpha}_i(t)=1~\forall\alpha_i\in S\}$ and $\mathbb{T}(S)^\circ:=\{t\in\mathbb{T}~|~\tilde{\alpha}_i(t)=1~\forall\alpha_i\in S\text{ and }\tilde{\alpha}_i(t)\ne1~\forall\alpha_i\notin S\}$.
\end{Definition}
Since $\mathbb{T}(S)=\bigcap_{\alpha\in S}\ker\tilde{\alpha}$, it is a subtorus of $\T$. Moreover $\mathbb{T}(S)^\circ\subset\mathbb{T}(S)$ is a Zariski open subset.
Notice that $\T(S)$ is disconnected in general. We denote $C(S):=\T(S)/\T(S)_1$, where $\T(S)_1$ is the identity component. For $u\in C(S)$, we denote by $\T(S)_u$ the corresponding connected component. 
We can write $\mathbb{T}(S)=\sqcup_{u\in C(S)}\mathbb{T}(S)_u$ and $\mathbb{T}(S)^\circ=\sqcup_{u\in C(S)}\mathbb{T}(S)^\circ_u$ as disjoint union of their connected components, where 
$\mathbb{T}(S)^\circ_u:=\mathbb{T}(S)_u\cap\mathbb{T}(S)^\circ$. It is then clear that, for any given $\Phi$, we have 
$$\mathbb{T}=\bigsqcup_{S\subset\Phi}\mathbb{T}(S)^\circ=\bigsqcup_{S\subset\Phi}\bigsqcup_{u\in C(S)}\mathbb{T}(S)^\circ_u.$$
\begin{Proposition}\label{admissispreserved1} Fix $S\subset\Phi$. Let $t\in\mathbb{T}(S)^\circ_u$ be a $\Phi$-admissible element. Then any $t'\in\mathbb{T}(S)^\circ_u$ is $\Phi$-admissible.
\end{Proposition}
\begin{proof} By hypothesis of $\Phi$-admissibility, there exists $v=(v_1,\dots,v_n)\in V$ such that $\Exp(v)=t$ and  $\alpha_i(v)\notin\mathbb{Z}_{>0}~\forall i=1,\dots,r$.

Consider the exponential map $\Exp\colon V\to \mathbb{T}$. Then the inverse image $\Exp^{-1}(\mathbb{T}(S)_u)=\sqcup_kG_k$ is a disjoint union of countably many affine subspaces $G_k$ of $V$. Similarly $\Exp^{-1}(\mathbb{T}(S)^\circ_u)=\sqcup_kG^\circ_k$, where $G^\circ_k\subset G_k$ is the complement of countably many (locally finite) subspaces. In particular, $G_k^\circ$ is a connected open subset of $G_k$. We can now find $k$ such that $v\in G^\circ_k$. Moreover, given $t'\in\mathbb{T}(S)^\circ_u$, we can choose $v'=(v_1',\dots,v_n')\in G^\circ_k$ such that $\Exp(v')=t'$. We claim that $v'$ satisfies $\alpha_i(v')\notin\mathbb{Z}_{>0}~\forall i=1,\dots,n$.

Suppose that there exists $i=1,\dots,r$ such that $\alpha_i(v')\in\mathbb{Z}_{>0}$. Since $G^\circ_k$ is connected, there exists a continuous family $v_s\colon [0,1]\to G^\circ_k$ such that $v_0=v$ and $v_1=v'$. This implies that $\alpha_i(v_0)\notin\mathbb{Z}_{>0}$ but $\alpha_i(v_1)\in\mathbb{Z}_{>0}$. If we call $t_s=\Exp(v_s)$, then $\tilde{\alpha}_i(t_s)=\Exp(\alpha_i(v_s))=1$ for all $s$. However, this is impossible because $\alpha_i(v_s)$ is not constant.
\end{proof}
\begin{Corollary}\label{nonadmissispreserved1} Let $t\in\mathbb{T}(S)^\circ_u$ be a non $\Phi$-admissible element. Then any $t'\in\overline{\mathbb{T}(S)^\circ_u}$ is non $\Phi$-admissible, where $\overline{\mathbb{T}(S)^\circ_u}$ is the topological closure of $\mathbb{T}(S)^\circ_u$.
\end{Corollary}
\begin{proof} The statement follows from Proposition \ref{admissispreserved1} and the fact that non $\Phi$-admissibility is a closed property.
\end{proof}
\begin{Remark} Since $\mathbb{T}(S)_u$ is irreducible, $\mathbb{T}(S)^\circ_u\ne\emptyset$ implies $\overline{\mathbb{T}(S)^\circ_u}=\mathbb{T}(S)_u$.
\end{Remark}
\begin{Corollary} Let us suppose $\mathbb{T}(S)$ is connected. Then all $t\in\mathbb{T}(S)^\circ$ are $\Phi$-admissible.
\end{Corollary}
\begin{proof} This follows from Lemma \ref{lem:open}, Proposition \ref{admissispreserved1} and the fact that $1\in\overline{\mathbb{T}(S)^\circ}$. 
\end{proof}
In general, we have the following. 
\begin{Corollary}\label{connect1givadmis} All $t\in\mathbb{T}(S)^\circ_1$ are $\Phi$-admissible.
\end{Corollary}
Using the decomposition 
$$
\T=
\bigsqcup_{S\subset\Phi}\T(S)^\circ=
\bigsqcup_{S\subset\Phi}\bigsqcup_{u\in C(S)}\T(S)_u^\circ, 
$$
we have the following. 
\begin{Theorem}\label{writingofadmisnonadmiss} We have that
$$\Adm(\Phi)=\bigsqcup_{S\subset\Phi}\bigsqcup_{\stackrel{u\in C(S)~\exists t\in\mathbb{T}(S)^\circ_u}{\mbox{\tiny $\Phi$-admissible}}}\mathbb{T}(S)^\circ_u,$$
and
$$\NonAdm(\Phi)=\bigsqcup_{S\subset\Phi}\bigsqcup_{\stackrel{u\in C(S) ~\exists t\in\mathbb{T}(S)^\circ_u}{\mbox{\tiny non $\Phi$-admissible}}}\mathbb{T}(S)^\circ_u=$$
$$=\bigcup_{S\subset\Phi}\bigcup_{\stackrel{u\in C(S) ~\exists t\in\mathbb{T}(S)_u^\circ}{\mbox{\tiny non $\Phi$-admissible}}}\mathbb{T}(S)_u.$$
\end{Theorem}
\begin{Remark}
\label{writingofadmisnonadmissremrk} From the previous Theorem, we see that $\NonAdm(\Phi)$ is the union of subtori of $\mathbb{T}$. Hence it is a Zariski closed subset of $\T$. 
\end{Remark}
From here to the end of the section, let us fix $t\in\mathbb{T}$. 
\begin{Definition} Given $\Phi$ as before, we define $\Phi_t:=\{\alpha\in\Phi~|~\tilde{\alpha}(t)=1\}\subset\Phi$.
\end{Definition}
Directly from the previous Definition, we have the following
\begin{Lemma} 
Let $S\subset\Phi$. 
Then $t\in\mathbb{T}(S)^\circ$ if and only if $S=\Phi_t$.
\end{Lemma}
Let $t\in\T$. Then the structure of the subtorus 
$\T(\Phi_t)$ is closely related to the admissibility of $t$. 
\begin{Proposition} 
\label{prop:T1}
If $t\in\mathbb{T}(\Phi_t)_1$, then $t$ is $\Phi$-admissible.
\end{Proposition}
\begin{proof} By hypothesis, we obtain that $t\in\mathbb{T}(\Phi_t)^\circ_1$ and hence we conclude by Corollary \ref{connect1givadmis}.
\end{proof}
\begin{Corollary} If $\mathbb{T}(\Phi_t)$ is connected, then $t$ is $\Phi$-admissible.
\end{Corollary}
We will see later that, under certain assumption, the converse to 
Proposition \ref{prop:T1} holds (Theorem \ref{thm:positivity}). 

Since the notion of $\Phi$-admissibility concerns only the real part of $\alpha_i$, we can restrict our attention to the real torus $\mathbb{T}_\R:=(S^1)^n=\Exp(V_\R)$. In fact, using the linear algebraic fact ``A system of linear equations 
with real coefficients has real solutions if and only if it has complex 
solutions'', we can prove the following key fact. 
$$
\T_\R(S)_u^\circ:=
\T(S)_u^\circ\cap \T_\R\neq\emptyset
\Longleftrightarrow
\T(S)_u^\circ\neq\emptyset, 
$$
where $S\subset\Phi$. As we have already seen, $\Adm(\Phi)$ is a 
disjoint union of the subsets of the form $\T(S)_u^\circ$. 
The above fact shows that 
it is enough to consider the real torus $\T_\R$. 
$$
\Adm(\Phi)=\bigsqcup_{S\subset\Phi}\bigsqcup_{\stackrel{u\in C(S) ~\exists t\in\mathbb{T_\R}(S)^\circ_u}{\mbox{\tiny $\Phi$-admissible}}}\mathbb{T}(S)^\circ_u,$$

Notice that if we consider $v\in V$ such that $\Exp(v)=t$ then, 
\begin{equation}
\label{eq:equiv}
\alpha(v)\in\mathbb{Z}\iff 
\tilde{\alpha}(v)=1 \iff 
\alpha\in\Phi_t
\end{equation}
for $\alpha\in\Phi$. Hence $t\in\T_\R$ is $\Phi$-admissible if and only if 
there exists $v\in V_\R$ such that $\Exp(v)=t$ and $\alpha(v)\leq 0$ for all 
$\alpha\in\Phi_t$. Now it seems natural to consider the following cone 
$\mathcal{D}(\Phi_t)$. 

\begin{Definition} Let $S\subset\Phi$. Define 
$\mathcal{D}(S):=\{v\in V_\R~|~\alpha(v)\le0,~\forall\alpha\in S\}$. We also define $\mathcal{D}_0(S):=\{v\in V~|~\alpha(v)=0~\forall\alpha\in S\}$. Note that $\mathcal{D}_0(S)$ is a linear subspace and $\mathcal{D}_0(S)\subset\mathcal{D}(S)\subset V_\R$.
\end{Definition}
Using the notion of the cone $\mathcal{D}(S)$, we can rewrite the 
definition of the $\Phi$-admissibility as follows. 


\begin{Proposition}\label{admissopenconeintersect} Let $t\in\mathbb{T}_\R$ and choose an element $v_0\in V_\R$ such that $\Exp(v_0)=t$. Then $t$ is $\Phi$-admissible if and only if 
\begin{equation}
\label{eq:intersection}
(v_0+\Lambda)\cap\mathcal{D}(\Phi_t)\ne\emptyset. 
\end{equation}
\end{Proposition}
\begin{proof}
It is obvious from $\Exp^{-1}(t)=v_0+\Lambda$. 
\end{proof}
Thus the admissibility is reduced to the existence of certain lattice 
points in a cone. The following gives a sufficient condition for $t\in\T_\R$ to be $\Phi$-admissible. 
\begin{Proposition} If $\R\cdot\mathcal{D}(\Phi_t)=V_\R$, then $t$ is $\Phi$-admissible, where $\R\cdot\mathcal{D}(\Phi_t)$ is the subspace generated by $\mathcal{D}(\Phi_t)$.
\end{Proposition}
\begin{proof} From the assumption, $\mathcal{D}(\Phi_t)$ is a cone in $V_\R$ which has a non-empty interior. Hence, $(v_0+\Lambda)\cap\mathcal{D}(\Phi_t)\ne\emptyset$, where $v_0\in V_\R$ is such that $\Exp(v_0)=t$.
\end{proof}
\begin{Proposition}\label{prop:indep} 
If $\Phi_t$ is composed of linearly independent elements, then $t$ is $\Phi$-admissible.
\end{Proposition}
\begin{proof} Since independence implies that $\mathcal{D}(\Phi_t)$ is full dimensional and so $(v_0+\Lambda)\cap\mathcal{D}(\Phi_t)\ne\emptyset$, where $v_0\in V_\R$ is such that $\Exp(v_0)=t$.
\end{proof}

\begin{Corollary}
If $\Phi$ is composed of linearly independent elements, then every $t\in\T$ is $\Phi$-admissible.
\end{Corollary}

We now give a characterization for 
admissibility in terms of the dual cone. 
\begin{Definition} 
Let $S\subset\Phi$. Define $\Cone(S):=\sum_{\alpha\in S}\R_{\geq 0}\cdot\alpha$. 
\end{Definition}
Notice that $\Cone(S)$ is a cone in $V^*$ and $\mathcal{D}(S)$ is its 
dual cone (times $(-1)$). 
Furthermore, from the self-duality of convex cones 
(\cite{fulton1993introductiontoric}), we have 
\begin{equation}
\label{eq:dualcone}
\mathcal{D}(S)=-\Cone(S)^\vee, \Cone(S)=-\mathcal{D}(S)^\vee. 
\end{equation}
The next lemmas will be used later. 
\begin{Lemma}\label{lemmaexpt1com} 
Let $S\subset\Phi$. Then $\Exp(\mathcal{D}_0(S))=\mathbb{T}_\R(S)_1$.
\end{Lemma}
\begin{proof} 
As we saw 
$$\Exp^{-1}(\T_\R(S)_1)=\{v\in V_\R~|~\alpha(v)\in\mathbb{Z}~\forall\alpha\in S\}$$ 
is a disjoint union of linear subspaces 
in $V_\R$. It is easily seen that one of the subspaces which passes through the origin 
$0\in V_\R$ is precisely equal to $\mathcal{D}_0(S)$. 
\end{proof}

\begin{Lemma}\label{positivityfromequal} Consider $S\subset\Phi$. Then $\mathcal{D}_0(S)=\mathcal{D}(S)$ if and only if there exist $c_\alpha>0$ for all $\alpha\in S$ such that $\sum_{\alpha\in S}c_\alpha \alpha=0$. 
\end{Lemma}
\begin{proof} Let us suppose that for all $\alpha\in S$ there exist $c_\alpha>0$ such that $\sum_{\alpha\in S}c_\alpha \alpha=0$. It is clear that $\mathcal{D}_0(S)\subset\mathcal{D}(S)$, and we will prove the opposite inclusion. 
Consider $v\in\mathcal{D}(S)$, then $\alpha(v)\le0$ for all $\alpha\in S$. Then $0=\sum_{\alpha\in\ S}c_\alpha \alpha(v)$ and since all coefficients are $c_\alpha>0$, this implies that $\alpha(v)=0$ and so $v\in\mathcal{D}_0(S)$.

Conversely, suppose that $\mathcal{D}_0(S)=\mathcal{D}(S)$. Then $\mathcal{D}(S)\subset V_\R$ is a linear subspace. Hence, $\Cone(S)=-\mathcal{D}(S)^\vee$ is also a linear subspace of $V^*_\R$. This implies that $-\alpha\in\Cone(S)$ for all $\alpha\in S$. Then, $-\alpha$ can be written as a linear combination $-\alpha=\sum_{\alpha'\in S\setminus\{\alpha\}}c_{\alpha\alpha'}\alpha'$, where $c_{\alpha\alpha'}\ge0$. Let us define $c_{\alpha\alpha}:=1$. Then, for every $\alpha\in S$, we have a linear relation $\sum_{\alpha'\in S}c_{\alpha\alpha'}\alpha'=0$, with $c_{\alpha\alpha'}\ge0$ and $c_{\alpha\alpha}=1$. The sum $\sum_{\alpha\in S}\sum_{\alpha'\in S}c_{\alpha\alpha'}\alpha'=0$ is now the required relation.
\end{proof}


\begin{Theorem}
\label{thm:positivity} Let $t\in\T_\R$. 
Suppose that 
there exists $c_\alpha>0$ for all $\alpha\in \Phi_t$ such that $\sum_{\alpha\in \Phi_t}c_\alpha \alpha=0$. 
Then $t$ is $\Phi$-admissible if and only if $t\in\mathbb{T}_\R(\Phi_t)_1$.
\end{Theorem}
\begin{proof} 
Recall that $t$ is $\Phi$-admissible if and only if there exists $v\in\mathcal{D}(\Phi_t)$ such that $\Exp(v)=t$. By Lemma \ref{positivityfromequal}, we have $\mathcal{D}_0(\Phi_t)=\mathcal{D}(\Phi_t)$. 
Hence $t$ is $\Phi$-admissible if and only if 
there exists $v\in\mathcal{D}_0(\Phi_t)$ such that $\Exp(v)=t$. 
By Lemma \ref{lemmaexpt1com}, it is equivalent to $t\in\mathbb{T}_\R(\Phi_t)_1$.
\end{proof}

\section{Structure of admissible rank one local systems}
\label{sec:structure}

Let $\mathcal{A}=\{L_0,\dots, L_n\}$ be a line arrangement in $\PPP^2$ such 
that $L_i=\{f_i=0\}$, where $f_i$ is a defining linear form. 
Set $M:=M(\mathcal{A})=\PPP^2\setminus (L_0\cup\dots\cup L_n)$. 
A cohomology class $\alpha\in H^1(M,\C)$ is given by 
\begin{equation} 
\alpha=\sum_{j=0}^na_j\frac{df_j}{f_j}, 
\end{equation}
with $a_j\in\C$ and $\sum_{j=0}^na_j=0$. The flat connection 
$\nabla:=d-\alpha\wedge$ determines a local system $\LL$. 
Since the local system $\LL$ is depending only on the 
the monodromies $t=(t_0, \dots, t_n)\in(\C^*)^n$ where $t=\Exp(\alpha):=
(\Exp(a_0), \dots, \Exp(a_n))$, it is denoted by $\LL_t$. 



For any $p\in\PPP^2$, we denote $\mathcal{A}_p:=\{L\in\mathcal{A}~|~p\in L\}$.

\begin{Definition}\label{defadlocsys} A local system $\mathcal{L}_t$ as above is \emph{admissible} if there is a cohomology class $\alpha\in H^1(M,\C)$ such that $\Exp(\alpha)=t$, $a_j\notin\mathbb{Z}_{>0}$ and, for any point $p\in L_0\cup\dots\cup L_n$ of multiplicity at least $3$, one has
$$a(p):=\sum_{\{j\in\{0,\dots,n\}~|~L_j\in\mathcal{A}_p\}}a_j\notin\mathbb{Z}_{>0}.$$
\end{Definition}
For an admissible local system it is proved that 
(\cite{esnault1992cohomology, schechtman1994local}) 
\begin{equation}
\label{isolike1admiss}
H^*(M(\A), \LL_t)\simeq
H^*(H^\bullet(M(\A), \C),\alpha\wedge). 
\end{equation}



Notice that Definition \ref{defadlocsys} is a particular case of the Definition \ref{defphiadmiss}. To see that, it is enough to put 
$$V=H^1(M,\C)=\{(x_0,\dots,x_n)\in\C^{n+1}~|~\sum_{i=0}^nx_i=0 \},$$ 
$$\mathbb{T}(\A)=\Hom(H_1(M(\A), \Z),\C^*)=\{(t_0, \dots, t_n)\in
(\C^*)^{n+1}\mid t_0t_1\cdots t_n=1\},$$ and 
$$
\Phi=\{x_0,\dots,x_n\}\cup\{x(p)\mid p\in\mathbb{P}^2, 
\sharp\A_p\geq 3\},$$ 
where $x_i\colon V\to\C$ is a linear form defined 
by $(x_0,\dots,x_n)\mapsto x_i$ and $x(p)\colon V\to\C$ is defined by $(x_0,\dots,x_n)\mapsto x_{i_1}+\cdots +x_{i_r}$, where $\mathcal{A}_p=\{L_{i_1},\dots ,L_{i_r}\}$.

Now we can apply the results in the previous section. 
For example, we have the following. 

\begin{Theorem}\label{admissisopen} The set of all non-admissible rank one local systems is the union of subtori of $\mathbb{T}(\A)$. Moreover, the set of all admissible rank one local systems forms a Zariski open subset of the character torus.
\end{Theorem}
We denotes the set of admissible (resp. non-admissible) local systems by 
$\Adm(\A)$ (resp. $\NonAdm(\A)$). 



\section{Characteristic varieties and admissible local systems}
\label{sec:characteristic}

Given $\mathcal{A}=\{L_0,\dots,L_n\}\subset\PPP^2$ a line arrangement, its characteristic varieties are the jumping loci for cohomology with coefficients in $\C$-valued rank one local systems on $M$ defined by 
$$\mathcal{V}^p_k(\mathcal{A}):=\{t\in\mathbb{T}(\A)~|~\dim H^k(M,\LL_t)\ge p\},$$
 for all $k, p\ge0$. For every $k$, we have a descending filtration
 $$\mathcal{V}^0_k(\mathcal{A})\supset \mathcal{V}^1_k(\mathcal{A}) \supset\cdots\supset\mathcal{V}^p_k(\mathcal{A})\supset\cdots .$$
 We denote $\mathcal{V}^1_1(\mathcal{A})$ simply by $\mathcal{V}_1(\mathcal{A})$. For more details, see for example \cite{Suciu:2013fk}. 
 
 In the past few years a lot of work has been done in order to understand the structure of $\mathcal{V}_1(\mathcal{A})$. In \cite{arapura1997geometry}, Arapura prove that the characteristic variety is the union of (possibly torsion-translated) subtori of $(\C^*)^n$. In \cite{libgober2000cohomology}, Libgober and Yuzvinsky, prove that the non-translated components of the characteristic variety are determined combinatorially. In the following two Theorems,  we describe the relation between admissible local systems and the points of  $\mathcal{V}_1(\mathcal{A})$.

\begin{Theorem} Let $\mathcal{A}\subset\PPP^2$ be a line arrangement. Any rank one local system belonging to a local component of $\mathcal{V}_1(\mathcal{A})$ is admissible.
\end{Theorem}
\begin{proof} Recall that any local component $C$ of $\mathcal{V}_1(\mathcal{A})$ corresponds to a multiple point $p\in\PPP^2$ of $\mathcal{A}$. Without loss of generality, we can suppose $\mathcal{A}_p=\{L_0,\dots, L_r\}$, with $r\ge2$. Then we can describe $C$ as $C=C_p:=\{t\in\mathbb{T}(M)~|~\prod_{i=0}^r t_i=1\text{ and }t_j=1~\forall j\in\{r+1,\dots,n\}\}$. 

It is enough to choose $\alpha\in H^1(M,\C)$ such that $\alpha=(a_0,\dots,a_r,0\dots,0)$, with $\Exp(a_i)=t_i$, $\sum_{i=0}^ra_i=0$ and $a_i=0$ if $a_i\in\Z$ for all $i=0,\dots, r$. 
\end{proof}
There exists a similar notion to the one of admissibility for local systems and it is the one of $1$-admissibility.
\begin{Definition} A local system $\mathcal{L}_t$  is \emph{$1$-admissible} if there is a cohomology class $\alpha\in H^1(M,\C)$ such that $\Exp(\alpha)=t$ and 
$$\dim H^1(M,\mathcal{L}_t)=\dim H^1(H^\bullet(M,\C),\alpha\wedge).$$
\end{Definition}
It is clear from the isomorphisms \eqref{isolike1admiss}  that any admissible local system is $1$-admissible.

Now let us discuss non-admissibility of translated components. 
\begin{Theorem}\label{translnonadmiss} Let $\mathcal{A}=\{L_1,\dots,L_n\}\subset\PPP^2$ be a line arrangement. Any rank one local system belonging to a translated component of $\mathcal{V}_1(\mathcal{A})$ is non-admissible.
\end{Theorem}
\begin{proof} 
Let $C\subset\mathcal{V}_1(\mathcal{A})$ be a translated component, i.e., 
an irreducible component which does not contain $1$. 
Let $t\in C\setminus\{\mbox{other components}\}$ be a local system that belongs only to the translated component $C\subset\mathcal{V}_1(\mathcal{A})$. 
Let us assume now that it is admissible. Then there exists $\alpha=(a_0,\dots,a_n)\in H^1(M,\C)$ satisfying $\sum_{j=0}^na_j=0$ such that $\Exp(\alpha)=t$ and $H^1(M,\LL_t)=H^1(M,\mathcal{L}(\alpha))\cong H^1(A^\bullet,\alpha\wedge)$, where $A^\bullet=H^\bullet(M,\C)$. Notice that since $t\in\mathcal{V}_1(\mathcal{A})$, then $\dim H^1(A^\bullet,\alpha\wedge)\ge 1$.

Let us consider now $t_s=\Exp(s \alpha)$ for $s\in[0,1]$, then $t_0=1$ and $t_1=t$. Moreover, $t_s\notin\mathcal{V}_1(\mathcal{A})$ for all except finitely many $s$ and since admissibility is an open property, by Theorem \ref{admissisopen}, $t_s$ is admissible if $0<s\ll 1$.

Let us now fix $0<s_1\ll 1$ such that $t_{s_1}\notin\mathcal{V}_1(\mathcal{A})$. This implies that $H^1(M,\LL_{t_{s_1}})=0$ and $t_{s_1}$ is admissible. Hence 
$$0=H^1(M,\LL_{t_{s_1}})=H^1(M,\mathcal{L}(s_1\alpha)).$$ 
By \cite{libgober2000cohomology} Proposition 4.2, we have the following $$0=\dim H^1(M,\mathcal{L}(s_1\alpha))\ge\dim H^1(A^\bullet,(s_1\alpha)\wedge).$$ Now by \cite{libgober2000cohomology} Lemma 4.1, we have that $$H^1(A^\bullet,(s_1\alpha)\wedge)\cong H^1(A^\bullet,\alpha\wedge).$$ But this implies that $0=\dim H^1(A^\bullet,\alpha\wedge)$ and this is impossible. Hence $\LL_t$ is non-admissible. 

Since $C$ is irreducible, the Zariski open subset 
$C\setminus\{\mbox{other components}\}$ is dense. By 
Theorem \ref{admissisopen}, $\LL_t$ is non-admissible for any $t\in C$. 
\end{proof}

\begin{Remark}
Notice that the previous proof proves also that such local systems are non $1$-admissible.
\end{Remark}

\begin{Lemma}\label{pathofintersectinchar} 
Let $X_1,X_2\subset\mathcal{V}_1(\mathcal{A})$ be two components of $\mathcal{V}_1(\mathcal{A})$ and define $k_i:=\max\{k\in\mathbb{Z}~|~X_i\subset\mathcal{V}_1^k(\mathcal{A})\}$, where $i=1,2$. Suppose there exists $t=\Exp(\alpha)\ne1$ such that $t\in X_1\cap X_2$ and define $l:=\max\{k\in\mathbb{Z}~|~t\in\mathcal{V}_1^k(\mathcal{A})\}$. Consider the path $t_s=\Exp(s \alpha)$, where $s\in[0,1]$. Then $t_s\notin\mathcal{V}_1^l(\mathcal{A})$ for almost all $s\in[0,1]$.
\end{Lemma}
\begin{proof} Notice that $l\ge k_1+k_2$ by 
\cite[Prop. 6.9.]{bartolo2010characteristic}. Suppose by contradiction that $t_s\in\mathcal{V}_1^l(\mathcal{A})$ for all $s\in[0,1]$. Then there exists $X_3\subset\mathcal{V}_1^l(\mathcal{A})$ a irreducible component of $\mathcal{V}_1^l(\mathcal{A})$ such that $t_s\in X_3$ for all $s\in[0,1]$. This implies that $t\in X_1\cap X_3$, but then by \cite[Prop. 6.9]{bartolo2010characteristic}, $t\in \mathcal{V}_1^{l+k_1}(\mathcal{A})\subset\mathcal{V}_1(\mathcal{A})$ contradicting the definition of $l$.  
\end{proof}
\begin{Theorem} 
\label{thm:intersection}
Let $X_1,X_2\subset\mathcal{V}_1(\mathcal{A})$ be two components of $\mathcal{V}_1(\mathcal{A})$. Suppose that $t\in X_1\cap X_2$ with $t\ne 1$. 
Then $\LL_t$ is a non-admissible local system.
\end{Theorem}
\begin{proof} 
Let us suppose that $t$ is admissible. Let $l\geq0$ be such that $\dim H^1(M,\mathcal{L}_t)=l$, i.e. $l$ is the maximum integer such that $t\in\mathcal{V}_1^l(\mathcal{A})$. 
By \cite[Prop.6.9.]{bartolo2010characteristic}, $l\geq 2$.  
By admissibility, there exists $\alpha\in H^1(M,\C)$ such that $\Exp(\alpha)=t$ and $H^1(M,\LL_t)=H^1(M,\mathcal{L}(\alpha))\cong H^1(A^\bullet,\alpha\wedge)$, which is $l$-dimensional. 

Let us consider now $t_s=\exp(s \alpha)$ for $s\in[0,1]$, then $t_0=1$ and $t_1=t$. Moreover by Lemma \ref{pathofintersectinchar}, $t_s\notin\mathcal{V}_1^l(\mathcal{A})$ for all except finitely many $s$ and since admissibility is an open property, $t_s$ is admissible if $0<t\ll 1$.

Let us now fix $0<s_1\ll 1$ such that $t_{s_1}\notin\mathcal{V}_1^l(\mathcal{A})$. This implies that $\dim H^1(M,\C_{t_{s_1}})<l$ and $t_{s_1}$ is admissible. Hence
$$l>\dim H^1(M,\LL_{t_{s_1}})=\dim H^1(M,\mathcal{L}(t_1\alpha)).$$ 

By \cite{libgober2000cohomology} Corollary 4.3, we know that  $l>\dim H^1(M,\mathcal{L}(s_1\alpha))\ge\dim H^1(A^\bullet,(s_1\alpha)\wedge)$. Now by \cite{libgober2000cohomology} Lemma 4.1, we have that $H^1(A^\bullet,(s_1\alpha)\wedge)\cong H^1(A^\bullet,(\alpha)\wedge)$. But this implies that $l>\dim H^1(A^\bullet,\alpha\wedge)$ and this is impossible.
\end{proof}
Notice that the previous proof also answer a question of Dimca from \cite{dimca2009admissible} about the 1-admissibility of local systems at the intersection of non-local components of $\mathcal{V}_1(\mathcal{A})$.

\section{Examples}
\label{sec:examples}

\begin{Example}(Non-Fano plane) The arrangement $\mathcal{A}$ is the realisation of the celebrated non-Fano matroid. It has defining polynomial $Q=xyz(x-y)(x-z)(y-z)(x+y-z)$. A decone of $\mathcal{A}$ is depicted in the Figure below (the seventh line is at infinity). Its characteristic varieties were computed in \cite{cohen1999characteristic}. The variety $\mathcal{V}_1(\mathcal{A})\subset (\C^*)^7$ has $6$ local components, corresponding to triple points, and $3$ non-local components, $\Pi_1=\Pi(25|36|47), \Pi_2=\Pi(17|26|35), \Pi_3=\Pi(14|23|56)$, corresponding to braid sub-arrangements. The local components meet only at $1$, but the non-local components also meet at the point $\rho=(1,-1,-1,1,-1,-1,1)$. 
Hence by Theorem \ref{thm:intersection}, $\rho\in\NonAdm(\A)$. 

We can also prove the non-admissibility of $\rho$ by using Theorem 
\ref{thm:positivity}. Indeed, in this case, we have 
$\Phi=
\{a_1, \dots, a_7, a_{125}, a_{136}, a_{237}, a_{246}, a_{345}, a_{567}\}$ 
and 
$$
\Phi_\rho=
\{a_1, a_4, a_7, a_{125}, a_{136}, a_{237}, a_{246}, a_{345}, a_{567}\}, 
$$
where $a_{ijk}$ denotes $a_i+a_j+a_k$. 
Obviously, 
$$
a_1+a_4+a_7+a_{125}+a_{136}+a_{237}+a_{246}+a_{345}+a_{567}=
3a_{1234567}=0. 
$$
Hence the assumption of Theorem \ref{thm:positivity} is satisfied. 
Then by the direct computation, we can show 
$$
\T(\Phi_\rho)=\{1, \rho\}. 
$$
Thus $\rho$ is non-admissible. 

Furthermore, after long case-by-case computation, we can show 
$$
\NonAdm(\A)=\{\rho\}. 
$$


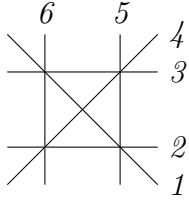
\begin{figure}[h]\label{nonfanofig}
\begin{center}
\begin{tikzpicture}
  \draw (-0.5,0) -- (1.5,0) node[anchor=west] {2};
  \draw (0,-0.5) -- (0,1.5) node[anchor=south] {6};
  \draw (-0.5,1) -- (1.5,1) node[anchor=west] {3};
  \draw (1,-0.5) -- (1,1.5) node[anchor=south] {5};
  \draw (-0.5,1.5) -- (1.5,-0.5) node[anchor=west] {1};
  \draw (-0.5,-0.5) -- (1.5,1.5) node[anchor=west] {4};
\end{tikzpicture}
\end{center}
\caption{The non-Fano arrangement}
\end{figure}
\end{Example}

Even though $\NonAdm(\A)\subset\T(\A)$ is combinatorially determined, 
it is in general very hard to give the precise description of 
$\NonAdm(\A)$. However, if we restrict our consideration to the essential 
open subset $\Tess(\A)$, then the situation becomes simple. 

\begin{Definition} Define
$$\Tess(\A):=\{(t_0, \dots, t_n)\in\T(\A)\mid t_i\neq 1, \forall 1\leq i\leq n\}.$$
\end{Definition}
We also denote $\Admess(\A)=\Adm(\A)\cap\Tess$ and 
$\NonAdmess(\A)=\NonAdm(\A)\cap\Tess$. There are several motivations 
to consider $\Tess(\A)$. First, D. Cohen \cite{coh-tri} proves that, 
under certain assumption, if $t_i=1$ for some $i$, then 
$H^*(M(\A), \LL_t)\simeq H^*(M(\A'), \LL_{t'})$, where 
$\A'=\A\setminus\{H_i\}$ and $t'=(t_1, \dots, t_{i-1}, t_{i+1}, \dots, t_n)$. 
In other words, the cohomology of non-essential local system 
can be, sometimes, computed by using smaller arrangement $\A'$. 
Secondly, the local system associated to non-trivial monodromies of 
the Milnor fiber is essential one \cite{Suciu:2013fk}. 

\begin{Example}(Deleted $B_3$ arrangement) Let $\mathcal{A}$ be the arrangement defined by $Q=xyz(x-y)(x-z)(y-z)(x-y-z)(x-y+z)$. It is obtained by deleting the plane $x+y-z=0$ from the arrangement $B_3$. The decone of $\mathcal{A}$, obtained by setting $z=1$, is depicted in the Figure below (the eighth line is at infinity). 

In this case, $\Phi=\{a_1, \dots, a_8, 
a_{136}, a_{147}, a_{128}, a_{235}, a_{246}, a_{348}, a_{5678}\}$. 
If $t\in\Tess(\A)$, then by definition, $\Phi_t\subset
\{a_{136}, a_{147}, a_{128}, a_{235}, a_{246}, a_{348}, a_{5678}\}$. 
These seven linear forms have a linear relation 
$$
a_{136}+
2a_{147}+
a_{128}+
2a_{235}+
a_{246}+
a_{348}+
2a_{5678}=
4a_{12345678}=0. 
$$
We can also prove that arbitrary six among the above seven linear forms are 
linearly independent. Thus if $\Phi_t\subsetneq 
\{a_{136}, a_{147}, a_{128}, a_{235}, a_{246}, a_{348}, a_{5678}\}$, 
by Proposition \ref{prop:indep}, $t$ is admissible. 
Hence we consider the case $\Phi_t=
\{a_{136}, a_{147}, a_{128}, a_{235}, a_{246}, a_{348}, a_{5678}\}$. 
Then $\T(\Phi_t)$ is a union of two $1$-dimensional tori $T_1$ and 
$T_2$, where 
\begin{eqnarray*}
T_1&=&\{(t,  t^{-1},  t^{-1}, t, t^2,  1, t^{-2},  1)\mid t\in\C^*\}, \\
T_2&=&\{(t, -t^{-1}, -t^{-1}, t, t^2, -1, t^{-2}, -1)\mid t\in\C^*\}. 
\end{eqnarray*}
Since $T_1$ is not essential, we have 
$\NonAdmess(\A)=T_2$. We note that $T_2$ is exactly equal to 
the translated component computed in 
\cite{suciu2002translated}. 


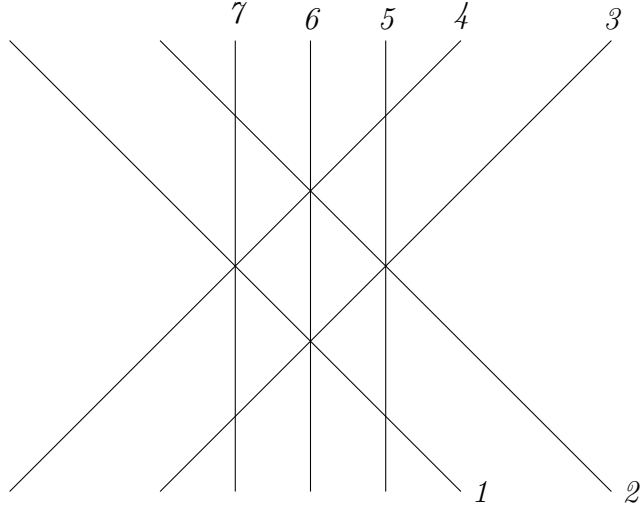
\begin{figure}[h]
\begin{center}
\begin{tikzpicture}
  \draw (0,-3) -- (0,3) node[anchor=south] {6};
  \draw (1,-3) -- (1,3) node[anchor=south] {5};
  \draw (-1,-3) -- (-1,3) node[anchor=south] {7};
  \draw (-2,3) -- (4,-3) node[anchor=west] {2};
  \draw (-4,3) -- (2,-3) node[anchor=west] {1};
\draw (-4,-3) -- (2,3) node[anchor=south] {4};
  \draw (-2,-3) -- (4,3) node[anchor=south] {3};
\end{tikzpicture}
\end{center}
\caption{The deleted $B_3$ arrangement}
\end{figure}
\end{Example}


\begin{Example}(Non-Pappus arrangement) Let $\mathcal{A}$ be the arrangement defined by $Q=xyz(x+y)(y+z)(x+3z)(x+2y+z)(x+2y+3z)(2x+3y+3z)$. A decone of $\mathcal{A}$ is depicted in the Figure below (the $9$-th line is at infinity). 
Let $t\in\Tess(\A)$. Then 
$$
\Phi_t\subset
\{a_{129}, a_{146}, a_{158}, a_{238}, a_{247}, 
a_{367}, a_{345}, a_{569}, a_{789}\}. 
$$
We can show that arbitrary eight linear forms among 
the above nine are linearly independent. Thus, if 
$|\Phi_t|<9$, then $t$ is admissible by Proposition \ref{prop:indep}. 
Hence we consider the case 
$$\Phi_t=\{a_{129}, a_{146}, a_{158}, a_{238}, a_{247}, a_{367}, a_{345}, a_{569}, a_{789}\}.$$
Then 
$$
\T(\Phi_t)=\{1, \rho, \rho^2, \rho^3, \dots, \rho^8\}, 
$$
where 
$\rho=(\zeta, \zeta, \zeta^4, \zeta, \zeta^4, \zeta^7, \zeta^7, \zeta^4, \zeta^7)$ with $\zeta=e^{2\pi\sqrt{-1}/9}$. 
The positive linear relation 
$$
a_{129}+a_{146}+a_{158}+a_{238}+a_{247}+
a_{367}+a_{345}+a_{569}+a_{789}=3a_{123456789}=0
$$
enable us to use Theorem \ref{thm:positivity}, and we can conclude 
$$
\NonAdmess(\A)=
\{\rho, \rho^2, \rho^3, \dots, \rho^8\}. 
$$

\begin{figure}[h]\label{nonpappusfig}
\begin{center}
\begin{tikzpicture}
  \draw (-3,3.5) -- (-3,-3.5) node[anchor=north] {1};
  \draw (0,3.5) -- (0,-3.5) node[anchor=north] {2};
  \draw (-4,0) -- (3.5,0) node[anchor=west] {8};
  \draw (-4,-1) -- (3.5,-1) node[anchor=west] {7};
  \draw (-3.5,3.5) -- (3.5,-3.5) node[anchor=north] {3};
  \draw (-4,1.5) -- (3.5,-2.25) node[anchor=west] {6};
  \draw (-4,0.5) -- (3.5,-3.25) node[anchor=south] {5};
  \draw (-4,5/3) -- (3.5,-10/3) node[anchor=west] {4};
\end{tikzpicture}
\end{center}
\caption{The non-Pappus arrangement}
\end{figure}
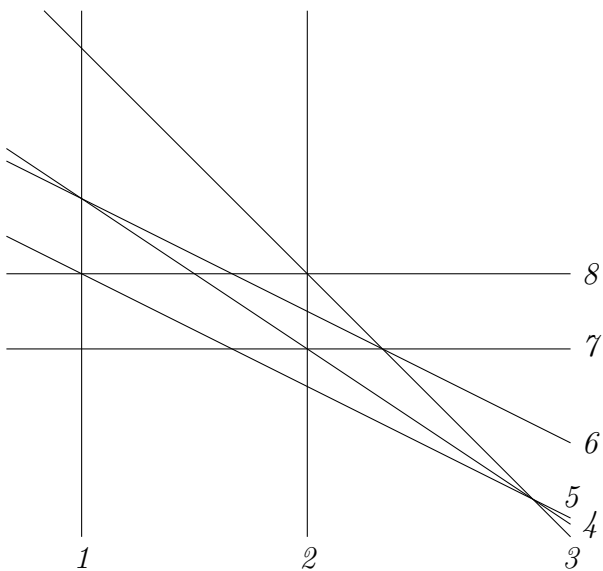
\end{Example}

\newpage

\bibliography{bibliothesis}{}

\begin{thebibliography}{10}

\bibitem{aom-kit}
K.~Aomoto and M.~Kita.
\newblock {\em Theory of hypergeometric functions. {W}ith an appendix by
  {T}oshitake {K}ohno. {T}ranslated from the {J}apanese by {K}enji {I}ohara.}
\newblock Springer Monographs in Mathematics. Springer-Verlag, Tokyo, 2011.

\bibitem{arapura1997geometry}
D.~Arapura.
\newblock Geometry of cohomology support loci for local systems {I}.
\newblock {\em J. Algebraic Geom.}, 6:563--597, 1997.

\bibitem{bartolo2010characteristic}
E.~A. Bartolo, J.~I. Cogolludo-Agust{\'\i}n, and D.~Matei.
\newblock Characteristic varieties of quasi-projective manifolds and orbifolds.
\newblock {\em arXiv preprint arXiv:1005.4761}, 2010.

\bibitem{coh-tri}
D.~Cohen.
\newblock Triples of arrangements and local systems.
\newblock {\em Proceedings of the American Mathematical Society},
  130(10):3025--3031, 2002.

\bibitem{cohen1999characteristic}
D.~Cohen and A.~I. Suciu.
\newblock Characteristic varieties of arrangements.
\newblock In {\em Mathematical Proceedings of the Cambridge Philosophical
  Society}, volume 127, pages 33--53. Cambridge Univ Press, 1999.

\bibitem{dimca2009admissible}
A.~Dimca.
\newblock On admissible rank one local systems.
\newblock {\em Journal of Algebra}, 321(11):3145--3157, 2009.

\bibitem{esnault1992cohomology}
H.~Esnault, V.~Schechtman, and E.~Viehweg.
\newblock Cohomology of local systems on the complement of hyperplanes.
\newblock {\em Inventiones mathematicae}, 109(1):557--561, 1992.

\bibitem{fulton1993introductiontoric}
W.~Fulton.
\newblock {\em Introduction to Toric Varieties.}
\newblock Number 131. Princeton University Press, 1993.

\bibitem{gel-hyp}
I.M. Gelfand.
\newblock General theory of hypergeometric functions.
\newblock In {\em Soviet Math. Dokl.}, volume~33, pages 573--577, 1986.

\bibitem{koh}
T.~Kohno.
\newblock Homology of a local system on the complement of hyperplanes.
\newblock {\em Proceedings of the Japan Academy, Series A, Mathematical
  Sciences}, 62(4):144--147, 1986.

\bibitem{libgober2000cohomology}
A.~Libgober and S.~Yuzvinsky.
\newblock Cohomology of the {O}rlik-{S}olomon algebras and local systems.
\newblock {\em Compositio mathematica}, 121(03):337--361, 2000.

\bibitem{nazir2009admissible}
S.~Nazir and Z.~Raza.
\newblock Admissible local systems for a class of line arrangements.
\newblock {\em Proceedings of the American Mathematical Society},
  137(4):1307--1313, 2009.

\bibitem{orlik1980combinatorics}
P.~Orlik and L.~Solomon.
\newblock Combinatorics and topology of complements of hyperplanes.
\newblock {\em Inventiones mathematicae}, 56(2):167--189, 1980.

\bibitem{rybnikov2011fundamental}
G.~L. Rybnikov.
\newblock On the fundamental group of the complement of a complex hyperplane
  arrangement.
\newblock {\em Functional Analysis and Its Applications}, 45(2):137--148, 2011.

\bibitem{schechtman1994local}
V.~Schechtman, H.~Terao, and A.~Varchenko.
\newblock Local systems over complements of hyperplanes and the {K}ac-{K}azhdan
  conditions for singular vectors.
\newblock {\em J. Pure Appl. Alg.}, 100(1--3):93--102, 1995.

\bibitem{suciu2002translated}
A.~I. Suciu.
\newblock Translated tori in the characteristic varieties of complex hyperplane
  arrangements.
\newblock {\em Topology and its Applications}, 118(1):209--223, 2002.

\bibitem{Suciu:2013fk}
A.~I. Suciu.
\newblock Hyperplane arrangements and {M}ilnor fibrations.
\newblock {\em arXiv preprint arXiv:1301.4851}, 2013.

\end{thebibliography}
\bibliographystyle{plain}

\end{document}